\documentclass[11pt,reqno]{amsart}
\usepackage{fullpage,times,graphicx,amssymb,amsmath,psfrag,xcolor,natbib}
\usepackage[colorlinks,citecolor=bbluegray,linkcolor=ddarkbrown,urlcolor=blue,breaklinks]{hyperref}
\usepackage{algorithmic,algorithm}

\usepackage[off]{auto-pst-pdf} % Switch to this if no figure recompilation necessary

\definecolor{ddarkbrown}{rgb}{0.5,0.2,0.05} \definecolor{bbluegray}{rgb}{0.05,0,0.5}

\newtheorem{theorem}{Theorem}[section]
\newtheorem{proposition}[theorem]{Proposition}
\newtheorem{definition}[theorem]{Definition}
\newtheorem{lemma}[theorem]{Lemma}
\newtheorem{corollary}[theorem]{Corollary}
\newtheorem{example}[theorem]{Example}
\renewenvironment{proof}{\textbf{Proof.}}{\QED\bigskip}

\newcommand{\BEAS}{\begin{eqnarray*}}
\newcommand{\EEAS}{\end{eqnarray*}}
\newcommand{\BEA}{\begin{eqnarray}}
\newcommand{\EEA}{\end{eqnarray}}
\newcommand{\BEQ}{\begin{equation}}
\newcommand{\EEQ}{\end{equation}}
\newcommand{\BIT}{\begin{itemize}}
\newcommand{\EIT}{\end{itemize}}
\newcommand{\BNUM}{\begin{enumerate}}
\newcommand{\ENUM}{\end{enumerate}}

\newcommand{\BA}{\begin{array}}
\newcommand{\EA}{\end{array}}

\newcommand{\ones}{\mathbf 1}

\newcommand{\reals}{{\mathbb R}}

\newcommand{\Expect}{\textstyle\mathop{\bf E}}

\newcommand{\QED}{~~\rule[-1pt]{6pt}{6pt}}

\newcommand{\argmin}{\mathop{\rm argmin}}

\newcommand{\dom}{\mathop{\bf dom}}

\renewcommand{\dom}{\mbox{dom}}
\let \oldsection \section
\renewcommand{\section}{\vspace{3ex plus 1ex}\oldsection}

\begin{document}
\title{Optimal Affine Invariant Smooth Minimization Algorithms}

\author{Alexandre d'Aspremont}
\address{CNRS \& D.I., UMR 8548, \vskip 0ex
\'Ecole Normale Sup\'erieure, Paris, France.}
\email{aspremon@ens.fr}

\author{Crist\'obal Guzm\'an}\footnote{Part of this work was done 
during a postdoctoral position at Centrum Wiskunde \& Informatica}
\address{Facultad de Matem\'aticas \& Escuela de Ingenier\'ia,
Pontificia Universidad Cat\'olica de Chile}
\email{crguzmanp@uc.cl}

\author{Martin Jaggi}
\address{Ecole Polytechnique F\'ed\'erale de Lausanne, Switzerland.}
\email{m.jaggi@gmail.com}

\keywords{}
\date{\today}
\subjclass[2010]{}

\begin{abstract}
We formulate an affine invariant implementation of the accelerated first-order algorithm in \citep{Nest83}. Its complexity bound is proportional to an affine invariant regularity constant defined with respect to the Minkowski gauge of the feasible set. We extend these results to more general problems, optimizing H\"older smooth functions using $p$-uniformly convex prox terms, and derive an algorithm whose complexity better fits the geometry of the feasible set and adapts to both the best H\"older smoothness parameter and the best gradient Lipschitz constant. Finally, we detail matching complexity lower bounds when the feasible set is an $\ell_p$ ball. In this setting, our upper bounds on iteration complexity for the algorithm in \citep{Nest83} are thus optimal in terms of target precision, smoothness and problem dimension.
\end{abstract}
\maketitle

\section{Introduction}\label{s:intro}
Here, we show how to implement the smooth minimization algorithm described in \citep{Nest83,Nest03} so that both its iterations and its complexity bound are invariant with respect to a change of coordinates in the problem. We focus on a generic convex minimization problem written
\BEQ\label{eq:min-pb}
\BA{ll}
\mbox{minimize } & f(x)\\
\mbox{subject to} & x \in Q,\\
\EA\EEQ
where $f$ is a convex function with Lipschitz continuous gradient and $Q$ is a compact convex set. Without too much loss of generality, we will assume that the interior of $Q$ is nonempty and contains zero. When $Q$ is sufficiently simple, in a sense that will be made precise later, \citet{Nest83} showed that this problem can be solved with a complexity $O(1/\sqrt{\varepsilon})$, where $\varepsilon$ is the target precision. Furthermore, it can be shown that this complexity bound is optimal in $\varepsilon$ for the class of smooth problems \citep{Nest03a}. 

While the dependence in $1/\sqrt{\varepsilon}$ of the complexity bound in \citet{Nest83} is optimal in $\varepsilon$, the various factors in front of that bound contain parameters which can heavily vary with implementation, i.e. the choice of norm and prox regularization function. In fact, the full upper bound on the iteration complexity of the optimal algorithm in \citep{Nest03a} is written
\[
\sqrt{\frac{8Ld(x^\star)}{\sigma\varepsilon}}
\]
where $L$ is the Lipschitz constant of the gradient, $d(x^\star)$ the value of the prox at the optimum and $\sigma$ its strong convexity parameter, all varying with the choice of norm and prox. This means in particular that, everything else being equal, this bound is not invariant with respect to an affine change of coordinates. 

Arguably then, the complexity bound varies while the intrinsic complexity of problem~\eqref{eq:min-pb} remains unchanged. Optimality in~$\varepsilon$ is thus no guarantee of computational efficiency, and a poorly parameterized optimal method can exhibit far from optimal numerical performance. On the other hand, optimal choices of norm and prox, hence of $L$ and $d$ should produce affine invariant bounds. Hence, affine invariance, besides its implications in terms of numerical stability, can also act as a guide to optimally choose norm and prox.

% Mention other affine invariant methods, essentially Frank-Wolfe.

Here, we show how to choose an underlying norm and a prox term for the algorithm in \citep{Nest83,Nest03} which make its iterations and complexity invariant by a change of coordinates. In Section~\ref{s:aff}, we construct the norm as the Minkowski gauge of centrally symmetric sets $Q$, then derive the prox using a definition of the regularity of Banach spaces used by \citep{Judi08b} to derive concentration inequalities. These systematic choices allow us to derive an affine invariant bound on the complexity of the algorithm in \citep{Nest83}.

When $Q$ is an $\ell_p$ ball, we show that this complexity bound is optimal for most, but not all, high dimensional regimes in $(n,\varepsilon)$. In Section~\ref{s:holder}, we thus extend our results to much more general problems, deriving a new algorithm to optimize H\"older smooth functions using $p$-uniformly convex prox functions. This extends the results of \citep{Nemi85} by incorporating adaptivity to the H\"older continuity of the gradient, and those of \citep{Nest15} by allowing general uniformly convex prox functions, not just strongly convex ones.

These additional degrees of freedom allow us to match optimal complexity lower bounds derived in Section~\ref{s:lp-balls} from \citep{Guzm13} when optimizing on $\ell_p$ balls, with adaptivity in the H\"older smoothness parameter and Lipschitz constant as a bonus. This means that, on $\ell_p$-balls at least, our complexity bounds are optimal not only in terms of target precision $\varepsilon$, but also in terms of smoothness and problem dimension. This shows that, in the $\ell_p$ setting at least, affine invariance does indeed lead to optimal complexity.

\section{Smooth Optimization Algorithm}\label{s:algo}
We first recall the basic structure of the algorithm in \citep{Nest83}. While many variants of this method have been derived, we use the formulation in \citep{Nest03}. We {\em choose a norm $\|\cdot\|$} and assume that the function $f$ in problem~\eqref{eq:min-pb} is convex with Lipschitz continuous gradient, so
\BEQ
\label{eq:lip-ineq}
f(y) \leq f(x) +\langle \nabla f(x),y-x\rangle +\frac{1}{2}L\|y-x\|^2, \quad x,y\in Q,
\EEQ
for some $L>0$. We also {\em choose a {prox} function} $d(x)$ for the set $Q$, i.e. a continuous, strongly convex function on $Q$ with parameter $\sigma$ (see \cite{Nest03a} or \cite{Hiri96} for a discussion of regularization techniques using strongly convex functions). We let $x_0$ be the center of $Q$ for the prox-function $d(x)$ so that
\[
x_0\triangleq\argmin_{x\in Q}d(x),
\]
assuming w.l.o.g. that $d(x_0)=0$, we then get in particular
\BEQ
\label{eq:d-strong-convex}
d(x)\geq\frac{1}{2}\sigma\|x-x_0\|^2.
\EEQ
We write $T_Q(x)$ a solution to the following subproblem
\BEQ
\label{eq:tq}
T_Q(x) \triangleq \argmin_{y\in Q}\left\{\langle \nabla f(x),y-x \rangle+\frac{1}{2}L\|y-x\|^2\right\}.
\EEQ
We let $y_0\triangleq T_Q(x_0)$ where $x_0$ is defined above. We recursively define three sequences of points: the current iterate $x_t$, the corresponding $y_t=T_Q(x_t)$, and the points
\BEQ
\label{eq:zk-def}
z_t \triangleq \argmin_{x \in Q}\left\{\frac{L}{\sigma}d(x)+\sum_{i=0}^t\alpha_i[ f(x_i)+\langle \nabla f(x_i),x-x_i\rangle]\right\}
\EEQ
given a step size sequence $\alpha_k\geq 0$ with $\alpha_0\in(0,1]$ so that
\BEQ
\label{eq:xyz-update}
\BA{l}
x_{t+1}=\tau_t z_t+(1-\tau_t)y_t\\
y_{t+1}=T_Q(x_{t+1})\\
\EA
\EEQ
where $\tau_t=\alpha_{t+1}/A_{t+1}$ with $A_t=\sum_{i=0}^t\alpha_i$. We implicitly assume here that $Q$ is simple enough so that the two subproblems defining $y_t$ and $z_t$ can be solved very efficiently. We have the following convergence result.

\begin{theorem}{\bf \cite{Nest03}.}\label{th:conv}
Suppose $\alpha_t=(t+1)/2$ with the iterates $x_t$, $y_t$ and $z_t$ defined in (\ref{eq:zk-def}) and (\ref{eq:xyz-update}), then for any $t\geq 0$ we have
\[
f(y_t)-f(x^{\star})\leq \frac{4Ld(x^{\star})}{\sigma\,(t+1)^2 }
\]
where $x^{\star}$ is an optimal solution to problem (\ref{eq:min-pb}).
\end{theorem}

If $\varepsilon>0$ is the target precision, Theorem \ref{th:conv} ensures that Algorithm~\ref{alg:smooth} will converge to an $\varepsilon$-accurate solution in no more than
\BEQ \label{eq:maxiter}
\sqrt{\frac{8Ld(x^\star)}{\sigma\varepsilon}}
\EEQ
iterations. In practice of course, $d(x^\star)$ needs to be bounded a priori and $L$ and $\sigma$ are often hard to evaluate.

\begin{algorithm}[ht]  
\caption{Smooth minimization.\label{alg:smooth}} 
\begin{algorithmic} [1]
\REQUIRE $x_0$, the prox center of the set $Q$.
\FOR{$t=0,\ldots,T$}
\STATE Compute $\nabla f(x_t)$.
\STATE Compute $y_t=T_Q(x_t)$.
\STATE Compute $z_t = \argmin_{x \in Q}\left\{\frac{L}{\sigma}d(x)+\sum_{i=0}^t\alpha_i[ f(x_i)+\langle \nabla f(x_i),x-x_i\rangle]\right\}$.
\STATE Set $x_{t+1}=\tau_t z_t + (1-\tau_t)y_t$.
\ENDFOR
\ENSURE $x_T,y_T \in Q$.
\end{algorithmic} 
\end{algorithm} 

While most of the parameters in Algorithm~\ref{alg:smooth} are set explicitly, the norm $\|\cdot\|$ and the prox function $d(x)$ are chosen arbitrarily. In what follows, we will see that a natural choice for both makes the algorithm affine invariant.

\section{Affine Invariant Implementation} \label{s:aff}
We can define an affine change of coordinates $x=Ay$ where $A\in\reals^{n \times n}$ is a nonsingular matrix, for which the original optimization problem in~\eqref{eq:min-pb} is transformed so
\BEQ\label{eq:aff-pb}
\BA{ll}
\mbox{minimize } & f(x)\\
\mbox{subject to} & x \in Q,\\
\EA
\qquad\mbox{ becomes}\qquad
\BA{ll}
\mbox{minimize } & \hat f(y)\\
\mbox{subject to} & y \in \hat Q,\\
\EA
\EEQ
in the variable $y\in\reals^n$, where 
\BEQ\label{eq:aff-chg}
\hat f(y) \triangleq f(Ay) 
\quad\mbox{and}\quad
\hat Q \triangleq A^{-1}Q.
\EEQ
Unless $A$ is pathologically ill-conditioned, both problems are equivalent and should have identical complexity bounds and iterations. In fact, the complexity analysis of Newton's method based on the self-concordance argument developed in \citep{Nest94} produces affine invariant complexity bounds and the iterates themselves are invariant. Here we will show how to choose the norm $\|\cdot\|$ and the prox function $d(x)$ to get a similar behavior for Algorithm~\ref{alg:smooth}.

\subsection{Choosing the Norm}\label{ss:norm}
We start by a few classical results and definitions. Recall that the {\em Minkowski gauge} of a set $Q$ is defined as follows.

\begin{definition}\label{def:mink}
Given $Q\subset \reals^n$ containing zero, we define the Minkowski gauge of $Q$ as
\[
\gamma_Q(x) \triangleq \inf\{ \lambda \geq 0 : x \in \lambda Q\}
\]
with $\gamma_Q(x)=0$ when $Q$ is unbounded in the direction $x$.
\end{definition}

When $Q$ is a compact convex, centrally symmetric set with respect to the origin and has nonempty interior, the Minkowski gauge defines a {\em norm}. We write this norm $\|\cdot\|_Q\triangleq \gamma_Q(\cdot)$. From now on, we will assume that the set $Q$ is centrally symmetric or use for example $\bar Q=Q - Q$ (in the Minkowski sense) for the gauge when it is not (this can be improved and extending these results to the nonsymmetric case is a classical topic in functional analysis). Note that any linear transform of a centrally symmetric convex set remains centrally symmetric. The following simple result shows why $\|\cdot\|_Q$ is potentially a good choice of norm for Algorithm~\ref{alg:smooth}.

\begin{lemma}\label{lem-lip-aff}
Suppose $f: \reals^n \rightarrow \reals$, $Q$ is a centrally symmetric convex set with nonempty interior and let $A\in\reals^{n \times n}$ be a nonsingular matrix. Then $f$ has Lipschitz continuous gradient with respect to the norm $\|\cdot\|_Q$ with constant $L>0$, i.e.
\[
f(y) \leq f(x) +\langle \nabla f(x),y-x\rangle +\frac{1}{2}L\|y-x\|_Q^2, \quad x,y\in Q,
\]
if and only if the function $\hat f(w) \triangleq f(Aw)$ has Lipschitz continuous gradient with respect to the norm $\|\cdot\|_{A^{-1}Q}$ with the same constant~$L$.
\end{lemma}
\begin{proof}
Let $w,y\in Q$, with $y=Az$ and $x=Aw$, then
\[
f(y) \leq f(x) +\langle \nabla f(x),y-x\rangle +\frac{1}{2}L\|y-x\|_Q^2, \quad x,y\in Q,
\]
is equivalent to 
\[
f(Az) \leq f(Aw) +\left\langle A^{-T} \nabla_{\!w} f(Aw), Az-Aw\right\rangle +\frac{1}{2}L\|Az-Aw\|_Q^2, \quad z,w\in A^{-1}Q,
\]
and, using the fact that $\|Aw\|_Q=\|w\|_{A^{-1}Q}$, this is also
\[
f(Az) \leq f(Aw) +\left\langle \nabla_{\!w} f(Aw),A^{-1}(Az-Aw)\right\rangle +\frac{1}{2}L\|z-w\|_{A^{-1}Q}^2, \quad z,w\in A^{-1}Q,
\]
hence the desired result.
\end{proof}

An almost identical argument shows the following analogous result for the property of \emph{strong convexity} with respect to the norm~$\|\cdot\|_Q$ and affine changes of coordinates. However, when starting from the above Lemma \ref{lem-lip-aff}, this can also be seen as a consequence of the well-known duality between smoothness and strong convexity (see e.g. \citep[Chap. X, Theorem 4.2.1]{Hiri96}).

\begin{theorem}  \label{thm:str_cvx_smooth}
Let $f:Q\to\reals$ be a convex l.s.c. function. Then $f$ is strongly convex w.r.t. norm $\|\cdot\|$ with constant $\mu>0$ if and only
$f^{\ast}$ has Lipschitz continuous gradient w.r.t. norm $\|\cdot\|_{\ast}$ with constant $L=1/\mu$. 
\end{theorem}

From the previous two results, we immediately have the following lemma.
\begin{lemma}\label{lem-strong-aff}
Suppose $f: \reals^n \rightarrow \reals$, $Q$ is a centrally symmetric convex set with nonempty interior and let $A\in\reals^{n \times n}$ be a nonsingular matrix. Suppose $f$ is strongly convex with respect to the norm $\|\cdot\|_Q$ with parameter $\sigma>0$, i.e.
\[
f(y) \geq f(x) +\langle \nabla f(x),y-x\rangle +\frac{1}{2}\sigma\|y-x\|_Q^2, \quad x,y\in Q,
\]
if and only if the function $\hat f(w) \triangleq f(Aw)$ is strongly convex with respect to the norm $\|\cdot\|_{A^{-1}Q}$ with the same parameter $\sigma$.
\end{lemma}

We now turn our attention to the choice of prox function in Algorithm~\ref{alg:smooth}.

\subsection{Choosing the Prox}\label{ss:prox}

Choosing the norm as $\|\cdot\|_Q$ allows us to define a norm without introducing an arbitrary geometry in the algorithm, since the norm is extracted directly from the problem definition. Notice furthermore that by Theorem \ref{thm:str_cvx_smooth} when $(\|\cdot\|_Q^2)^{\ast}$ is smooth,  we can set $d(x)=\|x\|_Q^2$.
%When $Q$ is smooth, a similar reasoning allows us to choose the prox term in Algorithm~\ref{alg:smooth}, and we can set $d(x)=\|x\|_Q^2$. 
The immediate impact of this choice is that the term $d(x^\star)$ in~\eqref{eq:maxiter} is bounded by one, by construction. This choice has other natural benefits which are highlighted below. We first recall a result showing that the conjugate of a squared norm is the squared dual norm.

\begin{lemma}\label{lem:conj-sqr}
Let $\|\cdot\|$ be a norm and $\|\cdot\|^*$ its dual norm, then
\[
\frac{1}{2}\left(\|y\|^*\right)^2=\sup_x ~y^Tx - \frac{1}{2} \|x\|^2.
\]
\end{lemma}
\begin{proof}
We recall the proof in \cite[Example\,3.27]{Boyd03} as it will prove useful in what follows. By definition, $x^Ty\leq \|y\|^*\, \|x\|$, hence
\[
y^Tx - \frac{1}{2} \|x\|^2 \leq \|y\|^*\, \|x\| - \frac{1}{2} \|x\|^2 \leq \frac{1}{2} \left(\|y\|^*\right)^2
\]
because the second term is a quadratic function of $\|x\|^2$, with maximum $\left(\|y\|^*\right)^2/2$. This maximum is attained by any $x$ such that $x^Ty= \|y\|^*\, \|x\|$ (there must be one by construction of the dual norm), normalized so $\|x\|=\|y\|^*$, which yields the desired result.
\end{proof}

Computing the prox-mapping in~\eqref{eq:tq} amounts to taking the conjugate of $\|\cdot\|^2$, so this last result (and its proof) shows that, in the unconstrained case, solving the prox mapping is equivalent to finding a vector aligned with the gradient, with respect to the Minkowski norm $\|\cdot\|_Q$. We now recall another simple result showing that the dual of the norm $\|\cdot\|_Q$ is given by~$\|\cdot\|_{Q^\circ}$ where $Q^\circ$ is the polar of the set $Q$.

\begin{lemma}\label{lem:dual-mink}
Let $Q$ be a centrally symmetric convex set with nonempty interior, then $\|\cdot\|_Q^* = \|\cdot\|_{Q^\circ}$.
\end{lemma}
\begin{proof}
We write
\BEAS
\|x\|_{Q^\circ} & = & \inf\{ \lambda \geq 0 : x \in \lambda Q^\circ\} =  \inf\{ \lambda \geq 0 : x^Ty\leq \lambda, \mbox{for all } y \in Q\}\\
& = & \inf\left\{ \lambda \geq 0 : \sup_{y\in Q}  x^Ty\leq \lambda\right\} = \sup_{y\in Q}  x^Ty = \|x\|_Q^*
\EEAS
which is the desired result.
\end{proof}

In the light of the results above, we conclude that whenever $Q^\circ$ is smooth %MJ: pointer for smoothness of sets would be good, as it was only used and defined for function so far
we obtain a natural prox function $d(x)=\|x\|_Q^2$, whose strong convexity parameter 
is controlled by the Lipschitz constant of the gradient of $\|\cdot\|_{Q^\circ}^2$.
%The last remaining issue to settle is the strong convexity of the squared Minkowski norm. Fortunately, this too is a classical result in functional analysis, as a %squared norm is strongly convex with respect to itself if and only if its dual norm has a smoothness modulus of power two. 
However, this does not cover the case where the squared norm $\|\cdot\|_Q$ is not strongly convex. In that scenario, we need to pick the norm based on $Q$ but find a strongly convex prox function not too different from~$\|\cdot\|^2_Q$. This is exactly the dual of the problem studied by \cite{Judi08b} who worked on concentration inequalities for vector-valued martingales and defined the regularity of a Banach space $(\mathbb{E},\|\cdot\|_\mathbb{E})$ in terms of the smoothness of the best smooth approximation of the norm $\|.\|_\mathbb{E}$. 

We first recall a few more definitions, and we will then show that the regularity constant defined by \cite{Judi08b} produces an affine invariant bound on the  term $d(x^\star)/\sigma$ in the complexity of the smooth algorithm in \citep{Nest83}.

\begin{definition} \label{def:bm}
Suppose $\|\cdot\|_X$ and $\|\cdot\|_Y$ are two norms on a space $\mathbb{E}$, the \emph{distortion} $d(\|\cdot\|_X,\|\cdot\|_Y)$ between these two norms is equal to the smallest product $ab>0$ such that
\[
\frac{1}{b} \|x\|_Y \leq \|x\|_X \leq a \|x\|_Y 
\]
over all $x\in \mathbb{E}$.
\end{definition}

Note that $\log d(\|\cdot\|_X,\|\cdot\|_Y)$ defines a metric on the set of all symmetric convex bodies in $\reals^n$, called the {\em Banach-Mazur distance}. We then recall the regularity definition in \citep{Judi08b}.

\begin{definition} \label{def:reg}
The \emph{regularity constant} of a Banach space $(\mathbb{E},\|.\|)$ is the smallest constant $\Delta>0$ for which there exists a smooth norm $p(x)$ such that 
\begin{enumerate}
\item $p(x)^2/2$ has a Lipschitz continuous gradient with constant $\mu$ w.r.t. the norm $p(x)$, with $1\leq\mu\leq \Delta$,
\item the norm $p(x)$ satisfies
\BEQ\label{eq:disto-reg}
\|x\|^2 \leq p(x)^2 \leq \frac{\Delta}{\mu} \|x\|^2, \quad \mbox{for all $x\in \mathbb{E}$}
\EEQ
hence $d(p(x),\|.\|)\leq \sqrt{\Delta/\mu}$.
\end{enumerate}
\end{definition}

Note that in finite dimension, since all norms are equivalent to the Euclidean norm with distortion at most $\sqrt{\dim \mathbb{E}}$, we know that all finite dimensional Banach spaces are at least $(\dim \mathbb{E})$-regular. Furthermore, the regularity constant is invariant with respect to an affine change of coordinates since both the distortion and the smoothness bounds are. % AA: write an explicit lemma here.
We are now ready to prove the main result of this section. 

\begin{proposition}\label{prop:reg-bound}
Let $\varepsilon>0$ be the target precision, suppose that the function $f$ has a Lipschitz continuous gradient with constant $L_Q$ with respect to the norm $\|\cdot\|_Q$ and that the space $(\reals^n,\|\cdot\|_Q^*)$ is $\Delta_Q$-regular, then Algorithm~\ref{alg:smooth} will produce an $\varepsilon$-solution to problem~\eqref{eq:min-pb} in at most
\BEQ\label{eq:aff-bnd}
\sqrt{\frac{4L_Q \Delta_Q}{\varepsilon}}
\EEQ
iterations. The constants $L_Q$ and $\Delta_Q$ are affine invariant.
\end{proposition}
\begin{proof}
If $(\reals^n,\|\cdot\|_Q^*)$ is $\Delta_Q$-regular, then by Definition~\ref{def:reg}, there exists a norm $p^*(x)$ such that $p^*(x)^2/2$ has a Lipschitz continuous gradient with constant~$\mu$ with respect to the norm $p^*(x)$, and \citep[Prop.\,3.2]{Judi08b} shows by conjugacy that the prox function $d(x)\triangleq p(x)^2/2$ is strongly convex with respect to the norm $p(x)$ with constant $1/\mu$. Now~\eqref{eq:disto-reg} means that
\[
\sqrt{\frac{\mu}{\Delta_Q}}~ \|x\|_Q \leq p(x) \leq \|x\|_Q, \quad \mbox{for all $x\in Q$}
\]
since $\|\cdot\|^{**}=\|\cdot\|$, hence
\BEAS
d(x+y) & \geq & d(x) + \langle \partial d(x), y\rangle + \frac{1}{2\mu}p(y)^2\\
& \geq & d(x) + \langle \partial d(x), y\rangle + \frac{1}{2\Delta_Q}\|y\|_Q^2\\
\EEAS
so $d(x)$ is strongly convex with respect to $\|\cdot\|_Q$ with constant $\sigma=1/\Delta_Q$, and using~\eqref{eq:disto-reg} as above
\[
\frac{d(x^\star)}{\sigma} = \frac{p(x^\star)^2 \Delta_Q}{2} \leq \frac{\|x^\star\|_Q^2 \Delta_Q}{2} \leq \frac{\Delta_Q}{2}
\]
by definition of $\|\cdot\|_Q$, if $x^\star$ is an optimal (hence feasible) solution of problem~\eqref{eq:min-pb}. The bound in~\eqref{eq:aff-bnd} then follows from \eqref{eq:maxiter} and its affine invariance follows directly from affine invariance of the distortion and Lemmas~\ref{lem-lip-aff} and~\ref{lem-strong-aff}.
\end{proof}

In Section~\ref{s:lp-balls}, we will see that, when $Q$ is an $\ell_p$ ball, the complexity bound in~\eqref{eq:aff-bnd} is optimal for most, but not all, high dimensional regimes in $(n,\varepsilon)$ where $n$ is greater than $\varepsilon^{-1/2}$. In the section that follows, we thus extend Algorithm~\ref{alg:smooth} to much more general problems, optimizing H\"older smooth functions using $p$-uniformly convex prox functions (not just strongly convex ones). These additional degrees of freedom will allow us to match optimal complexity lower bounds, with adaptivity in the H\"older smoothness parameter as a bonus.

\section{H\"older Smooth Functions \& Uniformly Convex Prox}\label{s:holder}
We now extend the results of Section~\ref{s:aff} to problems where the objective $f(x)$ is H\"older smooth and the prox function is $p$-uniformly convex, with arbitrary $p$. 
This generalization is necessary to derive optimal complexity bounds for smooth convex optimization over $\ell_p$-balls when $p>2$, and will require some extensions of the ideas we presented for the standard analysis, which was based on a strongly convex prox.  We will consider a slightly different accelerated method, that can be seen as a combination
of mirror and gradient steps \cite{Alle14}. This variant of the accelerated 
gradient method is not substantially different however from the one used in the previous section, and its purpose is to make the step-size analysis more transparent. It is worth emphasizing that an interesting
byproduct of our method is the analysis of an adaptive step-size policy, which can exploit 
weaker levels of H\"older continuity for the gradient.

In order to motivate our choice of $p$-uniformly convex prox, we begin with an 
example highlighting how the difficult geometries of $\ell_p$-spaces 
when $p>2$ necessarily lead to weak (dimension-dependent) complexity bounds for any prox.

\begin{example} \label{ex:Prox}
Let $2< p\leq\infty$ and let $\mathcal{B}_{p}$ be the unit $p$-ball on $\reals^n$. 
Let $\Psi:\reals^n\to\reals$ be any strongly convex (prox) function w.r.t. $\|\cdot\|_p$
with constant 1, and suppose w.l.o.g. that $\Psi(0)=0$. We will prove that 
\[
\sup_{x\in\mathcal{B}_p} \Psi(x) \geq n^{1-2/p}/2.
\]
We start from point $x_0=0$ and choose a direction $e_{1+}\in\{e_1,-e_1\}$
so that $\langle\nabla \Psi(0),e_{1+}\rangle\geq 0$. By strong convexity, we have, for $x_1\triangleq x_0+\frac{e_1}{n^{1/p}}$, 
\[
\Psi(x_1)\geq \Psi(x_0)+\langle\nabla \Psi(0),e_{1+}\rangle+\frac12\|x_1-x_0\|^2
\geq \frac{1}{2n^{2/p}}.
\] 
Inductively, we can proceed adding coordinate vectors
one by one, $x_{i}\triangleq x_{i-1}+\frac{e_{i+}}{n^{1/p}}$, for $i=1,\ldots,n$, where 
$e_{i+}\in\{e_i,-e_i\}$ is chosen so that $\langle \nabla\Psi(x_{i-1}),e_{i+}\rangle
\geq 0$. For this choice we can guarantee
\[
\Psi(x_{i})\geq \Psi(x_{i-1})+\langle \nabla\Psi(x_{i-1}),e_{i+}\rangle+\frac{1}{2 n^{2/p}}\geq \dfrac{i}{2n^{2/p}}.
\]
At the end, the vector $x_n\in\mathcal{B}_p$ and $\Psi(x_n)\geq n^{1-2/p}/2$.
\end{example}

%Let $(\bE,\|\cdot\|)$ be a $d$-dimensional normed space, 
%and let ${\mathcal{B}}_{\|\cdot\|}$ be the unit ball of $\bE$ (when it is clear 
%from context we will drop
%the norm in the notation). As usual, the dual space 
%$(\bE^{\ast},\|\cdot\|_{\ast})$ is comprised of bounded linear functionals 
%acting on $\bE$, and for all $w\in\bE^{\ast}$, 
%$\|w\|_{\ast}\triangleq\sup_{x\in \mathcal{\mathcal{B}}_{\|\cdot\|}} \langle w,x\rangle$.
%Note that since $(\bE,\|\cdot\|)$ is reflexive, the double dual 
%operation takes us back to the primal, i.e. $\bE^{\ast\ast}=\bE$.

%Any convex body (symmetric, closed, bounded
%and full-dimensional) $Q\subseteq \bE$ induces a norm $\|\cdot\|_Q$.
%Its dual norm will be denoted by $\|\cdot\|_{Q^\ast}$. Let $Q\subseteq \bE$ be a closed convex set. For a proper
%convex function $f:X\to \bar\reals$, at any $x\in Q$ where its 
%subdifferential is nonempty, we will denote  $\nabla f(x)$ any
%element of the subgradient $\partial f(x)$. For $1\leq p \leq \infty$ we define the conjugate of $p$ as the number 
%$1\leq p_{\ast}\leq \infty$ such that $\frac1p+\frac{1}{p_{\ast}}=1$,
%so $p_{\ast}=\frac{p}{p-1}$. Let $\ell_p^n=(\reals^n,\|\cdot\|_p)$ be the
%space $\reals^n$ endowed with the standard $p$-norm. Then, its
%dual space is isometrically isomorphic to $\ell_{p_{\ast}}^n$.

\subsection{Uniform Convexity and Smoothness}
The previous example shows that strong convexity of the prox function is too restrictive
when dealing with certain domain geometries, such as $Q=\mathcal {B}_p$ when $p>2$. In order to obtain dimension-independent bounds for these cases, we will have to consider relaxed notions of regularity for the prox, namely $p$-uniform convexity and its dual notion of $q$-uniform smoothness. For simplicity, we will only consider the case of subdifferentiable convex functions, which suffices for our purposes.

\begin{definition}[Uniform convexity and uniform smoothness]
Let $2\leq p<\infty$, $\mu>0$ and $Q\subseteq \reals^n$, a closed convex set. A 
subdifferentiable function $\Psi:Q\to\reals$ is $p$-uniformly convex with constant
$\mu$ w.r.t. $\|\cdot\|$ iff for all $x\in\mathring{Q}$%MJ: notation not optimal in view of polar ^\circ
, $y\in Q$,
\BEQ \label{unif_conv_diff}
\Psi(y)\geq \Psi(x)+\langle \nabla \Psi(x),y-x\rangle +\dfrac{\mu}{p}\|y-x\|^p.
\EEQ
Now let $1<q\leq 2$ and $L>0$. A subdifferentiable function 
$\Phi:Q\to\reals$ is $q$-uniformly smooth with constant~$L$ w.r.t. $\|\cdot\|$ 
iff for all $x\in\mathring{Q}$, $y\in Q$,
\BEQ \label{unif_smooth_diff}
\Phi(y)\leq \Phi(x)+\langle \nabla \Phi(x),y-x\rangle +\dfrac{L}{q}\|y-x\|^q.
\EEQ
\end{definition}

From now on, whenever the constant $\mu$ of $p$-uniform convexity is not explicitly stated, 
$\mu=1$.
%%%%%%%%%%%%%%%%%%%%%%%%%%%%
We turn our attention to the question of how to obtain an affine invariant prox in the uniformly convex setup. 
In the previous section it was observed that the regularity constant of the dual space
provided such tuning among strongly convex prox functions, however we are not aware of extensions of this notion 
to the uniformly smooth setup. Nevertheless, the same purpose can be achieved by directly minimizing the growth 
factor among the class of uniformly convex functions, which leads to the following notion. 

\begin{definition}[Constant of variation]
Given a $p$-uniformly convex function $\Psi:\reals^n\to\reals$, we define its constant 
of variation on $Q$ as $D_{\Psi}(Q)\triangleq \sup_{x\in Q} \Psi(x)-\inf_{x\in Q} \Psi(x)$. Furthermore, we define 
\BEQ\label{def:DpQ}
D_{p,Q} \triangleq \inf_\Psi\left\{\, \sup_{x\in Q}\Psi(x) \,\, \bigg| \,\,
\Psi:Q\to \reals_+ \mbox{ is } p\mbox{-uniformly convex w.r.t. } 
\|\cdot\|,\, \Psi(0)=0 \right \}.
\EEQ
\end{definition}
%AA:  Need \Psi(0)=0?
Some comments are in order. First, for fixed $p$, the constant $D_{p,Q}$ provides 
the optimal constant of variation among $p$-uniformly convex functions over $Q$, which means that, by construction, $D_{p,Q}$ is affine-invariant. Second, Example \ref{ex:Prox} showed that when $2<p<\infty$, we have 
$D_{2,\mathcal{B}_p} \geq n^{1-2/p}/2$, and the function $\Psi(x)=\|x\|_2^2/2$ shows this bound is tight. We will
later see that $D_{p,\mathcal{B}_p}=1$, which is a major improvement for large dimensions. \citep[Prop.\,3.3]{Judi08b} also shows that $\Delta_{Q} \geq D_{2,Q} \geq c\, \Delta_{Q}$, where $\Delta_{Q}$ is the regularity constant defined in~\eqref{def:reg} and $c>0$ is an absolute constant, since $\Psi(x)$ is not required to be a norm here.

When $Q$ is the unit ball of a norm, a classical result by \cite{Pisi75} links the constant of variation in~\eqref{def:DpQ} above with the notion of martingale cotype. A Banach space $(\mathbb{E},\|.\|)$ has M-cotype $q$ iff there is some constant $C>0$ such that for any $T\geq 1$ and martingale difference $d_1,\ldots,d_T\in \mathbb{E}$ we have
\[
\left(\sum_{t=1}^T \Expect\left[\|d_t\|^q\right]\right)^{1/q} \leq C \Expect \left[\left\|\sum_{t=1}^T d_t \right\|\right]
\]
\citet{Pisi75} then shows the following result.
\begin{theorem}\citep{Pisi75}
A Banach space $(\mathbb{E},\|.\|)$ has M-cotype $q$ iff there exists a $q$ uniformly convex norm equivalent to $\|\cdot\|$.
\end{theorem}
In the same spirit, there exists a concrete characterization of a function achieving the optimal constant of variation, see e.g. \citep{Sreb11}. Unfortunately, this characterization does not lead to an efficiently computable prox. For the analysis of our accelerated method with uniformly convex prox, we will also need the notion of {\em Bregman divergence}.

\begin{definition}[Bregman divergence]\label{def:Bregman}
Let $(\reals^n,\|\cdot\|)$ be a normed space, and $\Psi:Q\to\reals$ be a 
$p$-uniformly convex function w.r.t. $\|\cdot\|$. 
We define the Bregman divergence as
\[V_x(y)\triangleq\Psi(y)-\langle \nabla \Psi(x),y-x\rangle -\Psi(x)
\quad \forall x\in \mathring{Q},\,\, \forall y\in Q.\]
Observe that $V_x(x)=0$ and $V_x(y)\geq\frac1p\|y-x\|^p.$
\end{definition}

For starters, let us prove a simple fact that will be useful in the 
complexity bounds.

\begin{lemma}\label{dgf_identity}
Let $\Psi:Q\to\reals$ be a $p$-uniformly convex function, and $V_x(\cdot)$ the corresponding
Bregman divergence. Then, for all $x$, $x^{\prime}$ and $u$ in $Q$
\[ 
V_{x}(u)-V_{x^{\prime}}(u)-V_{x}(x^{\prime}) = 
\langle \nabla V_{x}(x^{\prime}),u-x^{\prime}\rangle. 
\]
\end{lemma}
\begin{proof}
From simple algebra
\begin{eqnarray*}
&&V_{x}(u)-V_{x^{\prime}}(u)-V_{x}(x^{\prime}) \\ 
&=& \Psi(u)-\langle\nabla \Psi(x),u-x\rangle-\Psi(x)
			-[\Psi(u)-\langle\nabla \Psi(x^{\prime}),u-x^{\prime}\rangle-\Psi(x^{\prime})]
			-V_x(x^{\prime}) \\
			&=& \langle\nabla \Psi(x^{\prime})-\nabla \Psi(x),u-x^{\prime}\rangle
			+\underbrace{\Psi(x^{\prime})-\langle\nabla \Psi(x),x^{\prime}-x\rangle-\Psi(x)}_{=V_x(x^{\prime})}-V_x(x^{\prime}).\\
			&=& \langle\nabla \Psi(x^{\prime})-\nabla \Psi(x),u-x^{\prime}\rangle
			 = \langle \nabla V_{x}(x^{\prime}),u-x^{\prime}\rangle.
\EEAS
which is the desired result.
\end{proof}

\subsection{An Accelerated Method for Minimizing H\"older Smooth Functions}

We consider classes of weakly-smooth convex functions. For a H\"older exponent
$\sigma\in(1,2]$ %MJ: can we use sth else than sigma here? avoiding confusions with s.c. before
we denote the class ${\mathcal F}_{\|\cdot\|}^{\sigma}(Q,L_{\sigma})$ as the set
of convex functions $f:Q\to \reals$ such that for all $x,y\in Q$
\[ \|\nabla f(x) -\nabla f(y)\|_{\ast} \leq L_{\sigma} \,\|x-y\|^{\sigma-1}. \] 
Before describing the method, we first define
a step sequence that is useful in the algorithm. For a given $p$ such that $2\leq p <\infty$,
consider the sequence $(\gamma_t)_{t\geq 0}$ defined by $\gamma_1=1$ and
for any $t>1$, $\gamma_{t+1}$ is the major root of
\[
\gamma_{t+1}^{p} = \gamma_{t+1}^{p-1}+\gamma_t^{p}.
\]
This sequence has the following properties.

\begin{proposition} \label{prop:aux_seq}
The following properties hold for the auxiliary sequence $(\gamma_t)_{t\geq 0}$
\begin{enumerate}
\item The sequence is increasing.
\item $\gamma_t^p = \sum_{s=1}^t \gamma_s^{p-1}$.
\item $\frac{t}{p}\leq \gamma_t\leq t$.
\item $\sum_{s=1}^t \gamma_s^p\leq t \gamma_t^p$.
\end{enumerate}
\end{proposition}
\begin{proof} We get
\begin{enumerate}
\item By definition, $\gamma_{t+1}^{p} = \gamma_{t+1}^{p-1}+\gamma_t^{p}
\geq \gamma_t^p$, thus $\gamma_{t+1}\geq \gamma_t$.
\item By telescoping the recursion, $\gamma_t^p = \sum_{s=1}^t \gamma_s^{p-1}$.
\item For the lower bound, a Fenchel type inequality yields
$$ \gamma_t = \gamma_{t+1}^{1/p_{\ast}}[\gamma_{t+1}-1]^{1/p}
\leq \dfrac{\gamma_{t+1}}{p_{\ast}} +\dfrac{\gamma_{t+1}-1}{p}
=\gamma_{t+1}-\dfrac{1}{p}.$$
The upper bound
is proved by induction as follows
$$(t+1)^{p} = (t+1)^{p-1}+ t(t+1)^{p-1}
> (t+1)^{p-1} + t[t^{p-1}+(p-1)t^{p-2}] 
\geq (t+1)^{p-1}+ \gamma_t^p,$$
where the last inequality holds by induction hypothesis, $\gamma_t\leq t$.
As a conclusion, the major root defining $\gamma_{t+1}$ has to be
at most $t+1$.
\item By $(ii)$, we have,
$$\sum_{s=1}^t \gamma_s^p = \sum_{s=1}^t \sum_{r=1}^s \gamma_s^{p-1}
 = \sum_{r=1}^t (t-r) \gamma_r^{p-1} \leq t \sum_{r=1}^t \gamma_r^{p-1} 
 =t \gamma_t^p.$$
\end{enumerate}
which concludes the proof.
\end{proof}

We now prove a simple Lemma controlling the smoothness of $f$ terms of $\|\cdot\|^p$. This idea is a minor extension of the ``inexact gradient trick'' proposed in \citep{Devo11} and further studied in \citep{Nest15}, which corresponds to the special case where $p=2$ in the results described here. As in \citep{Devo11}, this trick will allow us to minimize H\"older smooth functions by treating their gradient as an inexact oracle on the gradient of a smooth function.

\begin{lemma}\label{lem:Mbound}
Let  $f\in {\mathcal F}_{\|\cdot\|}^{\sigma}(Q,L_{\sigma})$,
%be such that for all $0\leq \nu\leq 1$ and $x,y\in Q$
%$$ \| \nabla f(y)-\nabla f(x)\|_{\ast} \leq L_{\nu}\|y-x\|^{\nu}.$$
then for any $\delta>0$ and
\begin{equation} \label{eqn:MLowerBound}
M \geq  \left[\dfrac{2}{p}\left(\frac{p-\sigma}{\sigma}\right) \dfrac{1}{\delta} 
\right]^{\frac{p-\sigma}{\sigma}} L_{\sigma}^{\frac{p}{\sigma}}
\end{equation}
we have that for all $x,y\in Q$
$$ f(y) \leq f(x) +\langle \nabla f(x),y-x \rangle + \dfrac{1}{p} M \|y-x\|^p +\frac{\delta}{2}.$$
\end{lemma}
\begin{proof} 
By assumption on $f$, the following bound holds for any $x,y\in Q$
\begin{eqnarray*}
f(y) \leq f(x)+\langle \nabla f(x),y-x\rangle+ \dfrac{L_{\sigma}}{\sigma}\|y-x\|^{\sigma}.
\EEAS
Notice first that it suffices to show that for all $t\geq 0$
\begin{equation} \label{ineq:1D}
\dfrac{L_{\sigma}}{\sigma}t^{\sigma} \leq \dfrac{M}{p}t^p+\dfrac{\delta}{2}.
\end{equation}
This can be seen by letting $t=\|y-x\|$ and using \eqref{ineq:1D} 
in the preceding inequality. Let us prove \eqref{ineq:1D}. First recall the following Fenchel type inequality: if $r,s\geq 1$ and $1/r+1/s=1$ then for all $x$ and $y$ we have that
 $x y \leq \frac{1}{r} x^r+\frac{1}{s} y^s$. For $r=p/\sigma$,
 $s=p/(p-\sigma)$ and $x=t^{\sigma}$, we obtain 
\[
\dfrac{L_{\sigma}}{\sigma} t^{\sigma}\leq 
\dfrac{1}{p}\dfrac{L_{\sigma}}{y}t^p 
+ \dfrac{ L_{\sigma}(p-\sigma) }{ p\sigma} y^{ \frac{\sigma}{p-\sigma} }.
\]
Now we choose $y$ so that $\frac{\delta}{2}=\frac{ L_{\sigma}(p-\sigma) }{ p\sigma} y^{ \frac{\sigma}{p-\sigma} }$, which leads to the inequality
\[
\dfrac{L_{\sigma}}{\sigma} t^{\sigma}\leq 
\dfrac{1}{p}\dfrac{L_{\sigma}}{y}t^p +\dfrac{\delta}{2}.
\]
Finally, by our choice of $M$ we have that $M\geq L_{\sigma}/y$,
proving \eqref{ineq:1D} and therefore the result.
\end{proof}

\begin{algorithm}[hb]  
\caption{Accelerated Method with Bregman Prox \label{method:NN85}} 
\begin{algorithmic}[1]
\REQUIRE  $x_0\in Q$
%A first-order oracle for a function $f\in{\mathcal F}^{\sigma}_{\|\cdot\|}(Q,L_{\sigma})$, 
%and Bregman divergence $V_x(y)$ from $\Psi:Q\to\reals$ a $p$-uniformly 
%convex function w.r.t. $\|\cdot\|$ over $Q$, 
% \STATE $x^0\in Q$
\STATE $y_0 = x_0$, $z_0=x_0$, and $A_0=0$
\FOR{$t=0,\ldots,T-1$}
	   \STATE $\alpha_{t+1} = \gamma_{t+1}^{p-1}/M$
	   \STATE $A_{t+1}=A_t+\alpha_{t+1}$
	\STATE $\tau_t =\alpha_{t+1}/A_{t+1} $
	\STATE $x_{t+1}=\tau_t z_t+(1-\tau_t)y_t$
	\STATE Obtain from oracle $\nabla f(x_{t+1})$, and update
	\begin{eqnarray}
	y_{t+1} &=& \arg\min_{y\in Q}\left\{ \frac{M}{p}\|y-x_{t+1}\|^p
	+ \langle \nabla f(x_{t+1}),y-x_{t+1}\rangle \right\} \label{accGD}\\
	z_{t+1} &=& \arg\min_{z\in Q}\left\{ V_{z_t}(z)+\alpha_{t+1} \langle \nabla f(x_{t+1}),z-z_t\rangle \right\} \label{accMD}
	\end{eqnarray}	
\ENDFOR
\RETURN $y^T$
\end{algorithmic}
\end{algorithm}

As we will show below, the accelerated method described in Algorithm \ref{method:NN85} extends the $\ell_p$-setting for acceleration first proposed by \citep{Nemi85} to nonsmooth spaces, using Bregman divergences. This gives us more flexibility in the choice of prox function and allows us in particular to better fit the geometry of the feasible set.

\begin{proposition} \label{conv_IMD}
Let $f\in {\mathcal F}_{\|\cdot\|}^{\sigma}(Q,L_{\sigma})$ and $\Psi:Q\to \reals$
be $p$-uniformly convex w.r.t. $\|\cdot\|$. Then for any $\varepsilon>0$, setting
$\delta\triangleq\varepsilon/T$, and $M$ satisfying \eqref{eqn:MLowerBound},
the accelerated method in Algorithm \ref{method:NN85} guarantees an accuracy
\[
f(y^T)-f(y^*) \leq \dfrac{D_{\Psi}(Q)}{A_T} +\dfrac{\varepsilon}{2}
\]
after $T$ iterations.
\end{proposition}
\begin{proof}
Let $u\in Q$ be an arbitrary vector. Using the optimality conditions for subproblem~\eqref{accMD}, and Lemma~\ref{dgf_identity}, we get
\begin{eqnarray*}
\alpha_{t+1} \langle \nabla f(x_{t+1}),z_t-u\rangle 
	&  =  & \alpha_{t+1} \langle \nabla f(x_{t+1}),z_t-z_{t+1}\rangle 
		+ \alpha_{t+1} \langle \nabla f(x_{t+1}),z_{t+1}-u\rangle \\
	&\leq& \alpha_{t+1} \langle \nabla f(x_{t+1}),z_t-z_{t+1}\rangle 
		-  \langle \nabla V_{z_t}(z_{t+1}),z_{t+1}-u\rangle \\
	&  =  & \alpha_{t+1} \langle \nabla f(x_{t+1}),z_t-z_{t+1}\rangle 
		+V_{z_t}(u)-V_{z_{t+1}}(u)-V_{z_t}(z_{t+1}) \\
	&\leq& \left[\alpha_{t+1} \langle \nabla f(x_{t+1}),z_t-z_{t+1}\rangle 
		- \frac{1}{p}\|z_t-z_{t+1}\|^p\right] + V_{z_t}(u)-V_{z_{t+1}}(u).
\EEAS
Let us examine the latter term in brackets closely. For this, let 
$v=\tau_t z_{t+1}+(1-\tau_t)y_t$ and note that 
$x_{t+1}-v=\tau_t(z_t-z_{t+1})$. With $\tau_t=\alpha_{t+1}/A_{t+1}$ we also have, using Proposition~\ref{prop:aux_seq}\,(ii),
$$ \dfrac{1}{\tau_t^p} = \left( \dfrac{L \sum_{s=1}^{t+1} \gamma_s^{p-1}}{ L \gamma_{t+1}^{p-1} } \right)^p= \gamma_{t+1}^p = M A_{t+1}.  $$
From this we obtain 
\begin{eqnarray*}
\alpha_{t+1} \langle \nabla f(x_{t+1}),z_t-z_{t+1}\rangle - \frac{1}{p}\|z_t-z_{t+1}\|^p
 &=& \left\langle \frac{\alpha_{t+1}}{\tau_t} \nabla f(x_{t+1}), x_{t+1}-v\right\rangle 
 	-\frac{1}{p\tau_t^p} \|x_{t+1}-v\|^p \\
&=&  A_{t+1} \left[\langle \nabla f(x_{t+1}), x_{t+1}-v\rangle 
 	-\frac{M}{p}\, \|x_{t+1}-v\|^p\right] \\
&\leq& A_{t+1} \left[ \langle \nabla f(x_{t+1}), x_{t+1}-y_{t+1}\rangle 
 	-\frac{M}{p}\, \|x_{t+1}-y_{t+1}\|^p \right]\\
&\leq& A_{t+1} \,\left[f(x_{t+1})-f(y_{t+1})+\frac{\delta}{2}\right],
\EEAS
where the first inequality holds by the definition of $y_{t+1}$, and the
last inequality holds by Lemma \ref{lem:Mbound} and the choice of $M$. This means that
\begin{equation} \label{nabla_estimate}
\alpha_{t+1} \langle \nabla f(x_{t+1}),z_t-u\rangle \leq 
A_{t+1} \,[f(x_{t+1})-f(y_{t+1})+\delta/2]+ V_{z_t}(u)-V_{z_{t+1}}(u).
\end{equation}
From \eqref{nabla_estimate} and other simple estimates
\begin{eqnarray*}
&&\alpha_{t+1}[f(x_{t+1})-f(u)] \\
	&\leq& \alpha_{t+1}\langle\nabla f(x_{t+1}),x_{t+1}-u\rangle \\
	&=&  \alpha_{t+1}\langle\nabla f(x_{t+1}),x_{t+1}-z_t\rangle
		+\alpha_{t+1}\langle\nabla f(x_{t+1}),z_t-u\rangle\\
	&=&\dfrac{(1-\tau_t)\alpha_{t+1}}{\tau_t}\langle\nabla f(x_{t+1}),y_t-x_{t+1}\rangle
		+\alpha_{t+1}\langle\nabla f(x_{t+1}),z_t-u\rangle\\
	&\leq& \dfrac{(1-\tau_t)\alpha_{t+1}}{\tau_t}[f(y_t)-f(x_{t+1})]
		+\alpha_{t+1}\langle\nabla f(x_{t+1}),z_t-u\rangle\\
	&\leq&  \dfrac{(1-\tau_t)\alpha_{t+1}}{\tau_t}[f(y_t)-f(x_{t+1})]
		  +A_{t+1} [f(x_{t+1})-f(y_{t+1})+\delta/2] 
		  + V_{z_t}(u)-V_{z_{t+1}}(u)\\
	&  =  & (A_{t+1}-\alpha_{t+1}) [f(y_t)-f(x_{t+1})] 
		  +A_{t+1}[f(x_{t+1})-f(y_{t+1})+\delta/2] 
		   + V_{z_t}(u)-V_{z_{t+1}}(u).
\EEAS
Therefore
\[
A_{t+1} f(y_{t+1}) -A_t f(y_t) + V_{z_{t+1}}(u)-V_{z_t}(u) 
\leq \alpha_{t+1}f(u) + A_{t+1}\frac{\delta}{2}.
\]
Summing these inequlities, we obtain
\[ A_T f(y_T) +[V_{z_{t+1}}(u)-V_{z_0}(u)] \leq A_T f(u)+\sum_{t=1}^T A_t \frac{\delta}{2}.\]
Now, by Proposition \ref{prop:aux_seq}, we have 
$\dfrac{1}{A_T}\sum_{t=1}^T A_t \leq \dfrac{1}{\gamma_T^p} T\gamma_T^{p}
\leq T$, thus by the choice $\delta=\varepsilon/T$, we obtain
\[ 
f(y_T)- f(u) \leq \frac{V_{z_0}(u)}{A_T}+\dfrac{\varepsilon}{2}.
\]
Definition~\ref{def:Bregman} together with the fact that $\langle \nabla \Psi(x), y-x \rangle \geq 0$ when $x$ minimizes $\Psi(x)$ over $Q$ then yields the desired result.
\end{proof}

In order to obtain the convergence rate of the method, we need to
estimate the value of $A_T$ given the choice of $M$. For this we assume the
bound in \eqref{eqn:MLowerBound} is satisfied with equality. Since
$A_T=\gamma_T^p/M$ we can use Proposition \ref{prop:aux_seq} (iii), so that
\begin{eqnarray*}
A_T &=& \gamma_T^p \left[\dfrac{p}{2}\left(\dfrac{\sigma}{p-\sigma}\right)\dfrac{\varepsilon}{T} \right]^{\frac{p-\sigma}{\sigma}} 
L_{\sigma}^{-\frac{p}{\sigma}} \\
&\geq& p^{-p} T^{p+1-\frac{p}{\sigma}} \varepsilon^{\frac{p}{\sigma}-1}
\left[\dfrac{p}{2}\left(\dfrac{\sigma}{p-\sigma}\right)\right]^{\frac{p-\sigma}{\sigma}} 
L_{\sigma}^{-\frac{p}{\sigma}}. \\
\EEAS
Notice that to obtain an $\varepsilon$-solution it suffices to have 
$A_T\geq 2D_{\Psi}(Q)/\varepsilon$. By imposing this lower bound on the
lower bound obtained for $A_T$ we get the following complexity estimate.

\begin{corollary} \label{cor:alg-2}
Let $f\in {\mathcal F}_{\|\cdot\|}^{\sigma}(Q,L_{\sigma})$ and $\Psi:X\to \reals$
be $p$-uniformly convex w.r.t. $\|\cdot\|$. Setting
$\delta\triangleq\varepsilon/T$, and $M$ satisfying \eqref{eqn:MLowerBound},
the accelerated method in Algorithm \ref{method:NN85} requires
\[ 
T < p\left[2^p\left(\dfrac{p-\sigma}{\sigma}\right)^{p-\sigma} 
\dfrac{ D_{\Psi}(Q)^{\sigma} L_{\sigma}^p }{ \varepsilon^{p} } 
\right]^{\frac{1}{(p+1)\sigma-p}}+1.
%2^{\frac{p}{\sigma p+\sigma -p}} p 
%\left(\dfrac{p}{\sigma}-1\right)^{\frac{p-\sigma}{\sigma p+\sigma-p}} 
%D_{\Psi}(Q)^{\frac{\sigma}{\sigma p+\sigma-p}}
%(L_{\sigma}\varepsilon^{-1})^{\frac{p}{\sigma p+\sigma-p}}
\]
iterations to reach an accuracy $\varepsilon$.
\end{corollary}

%It is worth noticing that
%, in terms of the parameters $\varepsilon$, $D_{\Psi}(Q)$ and $L_{\sigma}$ 
%this complexity bound is unimprovable under a first-order oracle. 
We will later see that the algorithm above leads to optimal 
complexity bounds (that is, unimprovable up to constant factors),
for $\ell_p$-setups. However, our algorithm is highly sensitive to several
parameters, the most important being $\sigma$ (the smoothness) and $L_{\sigma}$ which sets the step-size. We now focus on 
designing an adaptive step-size policy, that does not require
$L_{\sigma}$ as input, and adapts itself to the best weak smoothness parameter $\sigma\in(1,2]$.

\subsection{An Adaptive Gradient Method}

\begin{algorithm}[h]  
\caption{Accelerated Method with Bregman Prox and Adaptive Stepsize \label{method:adaptAccel}} 
\begin{algorithmic}[1]
\REQUIRE  $x^0\in Q$
\STATE Set $y_0 = x_0$, $z_0=x_0$, $M_0=1$ and $A_0=0$.
\FOR{$t=0,\ldots,T-1$}
		\STATE $M=M_{t}/2$
		\REPEAT
	   		\STATE\label{set_step} Set
			\BEAS
			M &=& 2M\\
			\alpha &=& \max\left\{a: M^{\frac{1}{p-1}} a^{p_{\ast}} -a =A_t \right\} \\
	   		A&=&A_t+\alpha \\
			\tau &=&\alpha/A \\
			x_{t+1}&=&\tau z_t+(1-\tau)y_t
			\EEAS
			\STATE Obtain $\nabla f(x_{t+1})$, and compute
			\[
			y_{t+1} = \arg\min_{y\in Q}\left\{ \frac{M}{p}\|y-x_{t+1}\|^p+ \langle \nabla f(x_{t+1}),y-x_{t+1}\rangle \right\}
			\]
			\UNTIL
			\BEQ\label{if_step}
			f(y_{t+1})\leq f(x_{t+1})+\langle \nabla f(x_{t+1}),y_{t+1}-x_{t+1}\rangle+\dfrac{M}{p}\|y_{t+1}-x_{t+1}\|^p+\dfrac{\tau \varepsilon}{2}
			\EEQ
			\STATE Set $M_{t+1}=M/2$, $\alpha_{t+1}=\alpha$, $A_{t+1} = A$, $\tau_t =\tau$.
			\STATE  Compute
			\BEAS
	z_{t+1} &=& \arg\min_{z\in Q}\left\{ V_{z_t}(z)+\alpha_{t+1} \langle \nabla f(x_{t+1}),z-z_t\rangle \right\}
	\EEAS	
\ENDFOR
\RETURN $y$
\end{algorithmic}
\end{algorithm}

We will now extend the adaptive algorithm in \citep[Th.3]{Nest15} to handle $p$-uniformly convex prox functions using Bregman divergences.  This new method with adaptive step-size policy is described as Algorithm \ref{method:adaptAccel}. From line \ref{set_step} in Algorithm \ref{method:adaptAccel} we get the 
following identities
\begin{eqnarray} 
A^{p-1} &=& \alpha^p M \label{ident1}\\
\dfrac{1}{\tau^p} &=& M A. \label{ident2}
\end{eqnarray}
These identities are analogous to the ones derived for the non-adaptive variant. For
this reason, the analysis of the adaptive variant is almost identical. There are a few 
extra details to address, which is what we do now. First, we need to show that the line-search procedure is feasible. That is, it always terminates in finite time. This is intuitively true from Lemma~\ref{lem:Mbound}, but 
let us make this intuition precise. From \eqref{ident1} and \eqref{ident2} we have
$$
M \tau^{\frac{p-\sigma}{\sigma}} = 
\dfrac{A^{p-1}}{\alpha^p} \left( \dfrac{\alpha}{A}\right)^{\frac{p}{\sigma}-1} 
= \dfrac{1}{\alpha}  \left( \dfrac{A}{\alpha} \right)^{p-\frac{p}{\sigma}} 
\geq \dfrac{1}{\alpha}.
$$
Notice that whenever the condition \eqref{if_step} of Algorithm
\ref{method:adaptAccel} is not satisfied, $M$ is increased by a factor two.  
Suppose the line-search does not terminate, then $ \alpha \to 0$. However,  
by Lemma~\ref{lem:Mbound}, the termination condition~\eqref{if_step} is guaranteed to be satisfied as soon as
$$ M \geq  \left[\dfrac{2}{p}\left(\frac{p-\sigma}{\sigma}\right) \dfrac{1}{\varepsilon \tau} 
\right]^{\frac{p-\sigma}{\sigma}} L_{\sigma}^{\frac{p}{\sigma}},$$
which is a contradiction with $\alpha\to0$.

To produce convergence rates, we need a lower bound on the sequence $A_t$. Unfortunately, the analysis in \citep{Nest15} only works when $p=2$, %(otherwise $\gamma<1/2$), 
we will thus use a different argument. First, notice that by the line-search rule
$$ 
\dfrac{M_{t+1}}{2} \leq 
\left[\dfrac{2}{p}\left(\frac{p-\sigma}{\sigma}\right) \dfrac{1}{\varepsilon \tau_t} 
\right]^{\frac{p-\sigma}{\sigma}} L_{\sigma}^{\frac{p}{\sigma}},
$$
from which we obtain
\BEAS
\alpha_{t+1}^p &=& \tau_t^p A_{t+1}^p = \dfrac{A_{t+1}^{p-1}}{M_{t+1}} \\
&\geq& A_{t+1}^{p-1} \frac12 \left[\dfrac{p}{2}\left(\frac{\sigma}{p-\sigma}\right) \varepsilon \tau_t \right]^{\frac{p}{\sigma}-1} L_{\sigma}^{-\frac{p}{\sigma}} \\
&\geq& \frac12\left[ \dfrac{\varepsilon p}{2} 
\left(\frac{\sigma}{p-\sigma}\right) \right]^{\frac{p-\sigma}{\sigma}} 
L_{\sigma}^{-\frac{p}{\sigma}}\, 
 A_{t+1}^{p-\frac{p}{\sigma}} \alpha_{t+1}^{\frac{p}{\sigma}-1}.
\EEAS
This allows us to conclude
$$ \alpha_{t+1}^{ \frac{(p+1)\sigma-p}{\sigma} } \geq
\frac12\left[ \dfrac{\varepsilon p}{2} 
\left(\frac{\sigma}{p-\sigma}\right) \right]^{\frac{p-\sigma}{\sigma}} 
L_{\sigma}^{-\frac{p}{\sigma}}\, 
A_{t+1}^{\frac{p\sigma-p}{\sigma}}, $$
which gives an inequality involving $\alpha_{t+1}$ and $A_{t+1}$
\[
\alpha_{t+1} \geq 
\left( 2^{ -\frac{\sigma}{(p+1)\sigma-p} }
\left[ \dfrac{\varepsilon p}{2} 
\left(\frac{\sigma}{p-\sigma}\right) \right]^{ \frac{p-\sigma}{(p+1)\sigma-p} }
L_{\sigma}^{ -\frac{p}{(p+1)\sigma-p} } \right)
\, A_{t+1}^{ \frac{p\sigma-p}{(p+1)\sigma-p} }.
\]
Here is where we need to depart from Nesterov's analysis, as the condition 
$\gamma\geq 1/2$ in that proof does not hold. Instead, we show the following bound.

\begin{lemma} \label{lem:rate_alpha}
Suppose $\alpha_t\geq0$, $\alpha_0=0$ and $A_t=\sum_{j=0}^t \alpha_j$, satisfy
\[
\alpha_t \geq \beta A_t^s
\]
for some $s\in[0,1[$ and $\beta\geq 0$. Then,
\[
A_t \geq ((1-s)\beta t)^\frac{1}{1-s}
\]
for any $t\geq 0$.
\end{lemma}
\begin{proof}
The sequence $A_t$ follows the recursion $A_t-A_{t-1}\geq \beta A_t^s$. The function $h(x)\triangleq x-\beta x^s$ satisfies $h(0)=0$, $h'(0^+)<0$ and $h'(x)$ only has a single positive root. Hence, when $A_{t-1}>0$, the equation
\[
A_t - \beta A_t^s = A_{t-1}
\]
in the variable $A_t$ only has a single positive root, after which $h(A_t)$ is increasing. This means that to get a lower bound on $A_t$ it suffices to consider the extreme case of the sequence satisfying
\[
A_t - A_{t-1} = \beta A_t^s.
\]
Because $A_t$ is increasing, the sequence $A_t-A_{t-1}$ is increasing, hence there exists an increasing, convex, piecewise affine function $A(t)$ that interpolates $A_t$, whose breakpoints are located at integer values of $t$. By construction, this function $A(t)$ satisfies
\[
A^{\prime}(t) = A_{\lfloor t+1 \rfloor}-A_{\lfloor t\rfloor} = \alpha_{\lfloor t+1\rfloor} \geq \beta A_{\lfloor t+1 \rfloor}^s \geq \beta A(t)^s
\]
for any $t\notin \mathbb{N}$. In particular, the interpolant satisfies
\BEQ\label{eq:diff-ineq}
A'(t) \geq \beta A^s(t)
\EEQ
for any $t\geq 0$. Note that $1/A^s(t)$ is a convergent integral around 0, as $A(t)$ is linear around 0, and $A^{\prime}(\cdot)$ can be defined as a right continuous nondecreasing function, which is furthermore constant around 0; therefore the involved functions are integrable, and the Theorem of change of variables holds. Integrating the differential inequality
we get 
\[
\beta t \leq \int_{0}^t \dfrac{A^{\prime}(t)}{A^s(t)} dt = \int_{0}^{A(t)} \dfrac{du}{u^s}= \dfrac{A(t)^{1-s}}{1-s},
\]
yielding the desired result.
\end{proof}

Using Lemma~\ref{lem:rate_alpha} with $s=(p\sigma-p)/((p+1)\sigma-p)$ produces the following bound on $A_T$
\[
A_{T}\geq \frac12 \left(\frac{\sigma}{(p+1)\sigma-p}\right)^{\frac{(p+1)\sigma-p}{\sigma}} \left(\frac{\varepsilon p}{2}\frac{\sigma}{p-\sigma} \right)^{\frac{p-\sigma}{\sigma}} L_{\sigma}^{-\frac{p}{\sigma}} T^{\frac{(p+1)\sigma-p}{\sigma}}.
\]
To guarantee that $A_T\geq 2D_{\Psi}(Q)/\varepsilon$, it suffices to impose
\[ 
T \geq  C(p,\sigma) \left(\dfrac{ D_{\Psi}^{\sigma}(Q) L_{\sigma}^p }{\varepsilon^p} \right)^{\frac{1}{(p+1)\sigma-p}}
\]
where 
\[
C(p,\sigma)\triangleq \left( \dfrac{ (p+1)\sigma-p }{ \sigma } \right) \left(\frac{2(p-\sigma)}{p\sigma}\right)^{\frac{p-\sigma}{(p+1)\sigma-p}} 2^{\frac{2\sigma}{(p+1)\sigma-p}}.
\]

%\[
%\left(\frac{\sigma}{2}\right)^{\frac{p}{\sigma}}\frac{1}{(p+1)\sigma-p} \left(\varepsilon \frac{p}{p-\sigma} \right)^{\frac{p}{\sigma}-1} L_{\sigma}^{-\frac{p}{\sigma}} T^{\frac{(p+1)\sigma-p}{\sigma}} \geq \frac{2D_{\Psi}(Q)}{\varepsilon}.
%\]
%Re-arranging terms, and letting 
%\[
%C(p,\sigma)\triangleq 2^{\frac{\sigma}{p}}\left(\frac{\sigma}{(p+1)\sigma-p}\right)^{-(\sigma-1+\sigma/p)}\left(\frac{p\sigma}{2(p-\sigma)}\right)^{1-\sigma/p}
%\]
%AA: 2x check this...
%we get
%$$ \varepsilon \leq C(p,\sigma) D_{\Psi}(Q)^{\frac{\sigma}{p}} L_{\sigma} T^{-\sigma(1+\frac1p-\frac{1}{\sigma})},$$
%which gives a bound on the number of iterations of the method.
\begin{corollary}\label{cor:bound-iter}
Let $f\in {\mathcal F}_{\|\cdot\|}^{\sigma}(Q,L_{\sigma})$ and $\Psi:X\to \reals$ is $p$-uniformly convex w.r.t. $\|\cdot\|$. Then the number of iterations required by Algorithm~\ref{method:adaptAccel} to produce a solution with accuracy $\varepsilon$ is bounded by
\[
T \leq \inf_{1<\sigma\leq 2} \left[C(p,\sigma)  \left(\dfrac{ D_{\Psi}^{\sigma}(Q) L_{\sigma}^p }{\varepsilon^p} \right)^{\frac{1}{(p+1)\sigma-p}} \right].
\]
\end{corollary}

From Corollary \ref{cor:bound-iter} we obtain the affine-invariant bound on 
iteration complexity. Given a centrally symmetric convex body $Q\subseteq\reals^n$,
we choose the norm as its Minkowski gauge $\|\cdot\|=\|\cdot\|_Q$, and 
$p$-uniformly convex prox as the minimizer defining the optimal $p$-variation
constant, $\sup_{x\in Q}\Psi(x)= D_{p,Q}$. With these choices, the iteration 
complexity is
\[
T \leq \inf_{1<\sigma\leq 2} 
\left[C(p,\sigma)  \left(\dfrac{ D_{p,Q}^{\sigma} L_{\sigma,Q}^p }{\varepsilon^p} \right)^{\frac{1}{(p+1)\sigma-p}} \right],
%\left[ \left(\dfrac{C(p,\sigma)L_{\sigma}}{\varepsilon} \right)^{ 1/\sigma } D_{p,Q}^{ 1/p } \right]^{ \frac{p\sigma}{p\sigma+\sigma-p} }.
\]
where $L_{\sigma,Q}$ is the H\"older constant of $f$ quantified in the Minkowski
gauge norm $\|\cdot\|_{Q}$. As a consequence, the bound above is affine-invariant,
since also $D_{p,Q}$ is 
affine-invariant by construction. Observe our iteration bound automatically adapts 
to the best possible weak smoothness parameter $\sigma\in(1,2]$; note however 
that an implementable algorithm requires an accuracy certificate in order to stop 
with this adaptive bound. These details are beyond the scope of this paper, but we
refer to \citep{Nest15} for details. Finally, we will see in what follows that this affine invariant bound also matches corresponding lower bounds when $Q$ is an $\ell_p$ ball.

%\subsection{An Implementable Stopping Criterion}

%Let $N_t$ the number of gradient evaluations by the algorithm after $t$ steps
%(with the convention $N_0=0$), then 
%\begin{eqnarray*}
%N_T &=& \sum_{t=0}^{T-1} [N_{t+1}-N_t] =\sum_{t=0}^{T-1} 
%\log\left(\dfrac{M_{t+1}}{M_t/2} \right)=T +\log(M_T)-\log(M_0)\\
%&\leq& T+\log\left(2\left[\dfrac{2}{p}\left(\dfrac{p}{1+\nu}-1 \right)\dfrac{T}{\varepsilon} \right]^{\frac{p}{1+\nu}-1} L_{\nu}^{\frac{p}{1+\nu}} \right).
%\EEAS

\section{Explicit Bounds on Problems over $\ell_p$ Balls}\label{s:lp-balls}

\subsection{Upper Bounds}
To illustrate our results, first consider the problem of minimizing a smooth convex function over the unit simplex, written
\BEQ\label{eq:min-simplex}
\BA{ll}
\mbox{minimize } & f(x)\\
\mbox{subject to} & \ones^T x\leq 1,\, x\geq 0,\\
\EA\EEQ
in the variable $x\in\reals^n$. 

As discussed in \citep[\S3.3]{Judi07}, choosing $\|\cdot\|_1$ as the norm and $d(x)=\log n+\sum_{i=1}^n x_i \log x_i$ as the prox function, we have $\sigma=1$ and $d(x^\star)\leq \log n$, which means the complexity of solving~\eqref{eq:min-simplex} using Algorithm~\ref{alg:smooth} is bounded by
\BEQ\label{eq:l1-complex}
\sqrt{8\frac{L_1 \log n}{\varepsilon}}
\EEQ
where $L_1$ is the Lipschitz constant of $\nabla f$ with respect to the $\ell_1$ norm. This choice of norm and prox has a double advantage here. First, the prox term $d(x^\star)$ grows only as $\log n$ with the dimension. Second, the $\ell_\infty$ norm being the smallest among all $\ell_p$ norms, the smoothness bound $L_1$ is also minimal among all choices of $\ell_p$ norms. 

Let us now follow the construction of Section~\ref{s:aff}. The simplex $C=\{x\in\reals^n: \ones^Tx\leq 1, x \geq 0\}$ is not centrally symmetric, but we can symmetrize it as the $\ell_1$ ball. The Minkowski norm associated with that set is then equal to the $\ell_1$-norm, so $\|\cdot\|_Q=\|\cdot\|_1$ here. The space $(\reals^n,\|\cdot\|_\infty)$ is $2 \log n$ regular \citep[Example 3.2]{Judi08b} with the prox function chosen here as $\|\cdot\|_{\alpha}^2/2$, with $\alpha=2\log n/(2\log n -1)$. Proposition~\ref{prop:reg-bound} then shows that the complexity bound we obtain using this procedure is identical to that in~\eqref{eq:l1-complex}. A similar result holds in the matrix case.

\subsubsection{Strongly Convex Prox} We can generalize this result to all cases where $Q$ is an $l_p$ ball. When $p\in[1,2]$, \citep[Ex.\,3.2]{Judi07} shows that the dual norm $\|\cdot\|_{\frac{p}{p-1}}$ is $\Delta_p$ regular, with 
\[
\Delta_p=\inf_{2 \leq \rho < \frac{p}{p-1}} ~ (\rho-1)n^{\frac{2}{\rho}-\frac{2(p-1)}{p}}\leq \min\left\{\frac{p}{p-1},C \log n\right\}, 
\quad \mbox{when } p\in[1,2].
\]
When $p\in[2,\infty]$, the regularity is only controlled by the distortion $d(\|\cdot\|_{\frac{p}{p-1}},\|\cdot\|_2)$, since $\|\cdot\|_\alpha$ is only smooth when $\alpha \geq 2$. This means that 
$\|\cdot\|_{\frac{p}{p-1}}$ is $\Delta_p$ regular, with
\[
\Delta_p=n^{\frac{p-2}{p}},\quad \mbox{when } p\in[2,\infty].
\]
This means that the complexity of solving 
\BEQ\label{eq:min_lp}
\BA{ll}
\mbox{minimize } & f(x)\\
\mbox{subject to} & x\in \mathcal{B}_p\\
\EA\EEQ
in the variable $x\in\reals^n$, where $\mathcal{B}_p$ is the $\ell_p$ ball, using Algorithm~\ref{alg:smooth}, is bounded by
\BEQ\label{eq:lp_simple}
\sqrt{\frac{4L_p \Delta_p}{\varepsilon}}
\EEQ
where $L_p$ is the Lipschitz constant of $\nabla f$ with respect to the $\ell_p$ norm. %We will see below that this bound is optimal up to polylogarithmic factors.
We will later see that this bound is nearly optimal when $1\leq p \leq 2$; however,
the dimension dependence on the bounds when $p>2$ is essentially suboptimal.
In order to obtain the optimal methods in this range we will need our $p$-uniformly convex extensions. 

\subsubsection{Uniformly Convex Bregman Prox}
In the case $2\leq p <\infty$, the function $\Psi_p(w)=\frac{1}{p}\|w\|_p^p$ is $p$-uniformly convex w.r.t. $\|\cdot\|_p$  (see, e.g. \citep{Ball94}), and thus 
$$ D_{p,\mathcal{B}_p}= 1, \quad \mbox{when } p\in[2,\infty].$$
As a consequence, Algorithm \ref{method:NN85} with $\Psi_p$ as $p$-uniformly requires 
\begin{eqnarray} \label{eq:lp_diff}
T \geq C(p) \left(\frac{L_p}{\varepsilon} \right)^{\frac{p}{p+2}}
\end{eqnarray}
iterations to reach a target precision $\varepsilon$, where $C(p)$ is 
a constant only depending on $p$ (which nevertheless diverges as $p\to\infty$). 
This complexity guarantee admits passage to the limit $p\to\infty$ with a 
poly-logarithmic extra factor. Note however that in this case we can avoid 
any dimension dependence by the much simpler Frank-Wolfe method.

\subsection{Lower Bounds}
We show that in the case of $\ell_p$ balls estimates from the proposed methods
are nearly optimal in terms of information-based complexity. 
We consider the class of problems given by the minimization of smooth convex
objectives with a bound $L_p$ on the Lipschitz constant of their gradients w.r.t. norm 
$\|\cdot\|_p$, and the feasible domain given the radius $R_p>0$ ball 
$\mathcal{B}_p(R)$. We emphasize that the lower bounds we 
present only hold for the large-scale regime, where the number of iterations 
$T$ is upper bounded by the dimension of the space, $n$. It is well-known
that when one can afford a super-linear (in dimension) number of iterations, 
methods such as the center of gravity or ellipsoid can achieve 
better complexity estimates \citep{Nemi79}.

First, in the range $1\leq p\leq 2$ we can immediately use the lower bound on risk from \citep{Guzm13}, 
\[
\Omega\left(\frac{L_p R_p^2}{T^2\log(T+1)}\right)
\]
where $T$ is the number of iterations, which translates into the following lower 
bound on iteration complexity
\[
\Omega\left(\sqrt{\frac{L_p R_p^2}{\varepsilon \log n}}\right)
\]
as a function of the target precision $\varepsilon>0$. Therefore, the affine invariant algorithm is optimal, up to poly-logarithmic factors, in this range.

For the second range, $2<p\leq\infty$, the lower bound states the accuracy 
after~$T$ steps is no 
better than
\[
\Omega\left( \frac{L_p R_p^2}{\min[p,\log n] \, T^{1+2/p}} \right),
\]
which translates into the iteration complexity lower bound
\[ 
\Omega\left( \left(\dfrac{L_p R_p^2}{\min[p,\log n]\varepsilon}\right)^{\frac{p}{p+2}} \right).
\]
For fixed $2\leq p<\infty$, this lower bound matches --up to constant factors-- our iteration complexity obtained for these setups. For the case $p=\infty$, our algorithm
also turns out to be optimal, up to polylogarithmic in the dimension factors.

\section{Numerical Results} \label{s:numres}
We now briefly illustrate the numerical performance of our methods on a simple problem taken from \citep{Nest15}. To test the adaptivity of Algorithm~\ref{method:adaptAccel}, we focus on solving the following {\em continuous Steiner problem}
\BEQ\label{eq:steiner}
\min_{\|x\|_2\leq 1} \sum_{i=1}^{m} \|x-x_i\|_q
\EEQ
in the variable $x\in\reals^n$, with parameters $x_i\in\reals^n$ for $i=1,\ldots,m$. The parameter $q\in[1,2]$ controls the H\"older continuity of the objective. We sample the points $x$ uniformly at random in the cube $[0,1]^n$. We set $n=50$, $m=10$ and the target precision $\varepsilon=10^{-12}$. We compare iterates with the optimum obtained using CVX \citep{Gran01}. We observe that while the algorithm solving the three cases $q=1, 1.5, 2$ is identical, it is significantly faster on smoother problems, as forecast by the adaptive bound in Corollary~\ref{cor:bound-iter}.

\begin{figure}[h]
\begin{center}
\begin{tabular}{cc}
\psfrag{obj}[b][t]{$f_t-f^*$}
\psfrag{it}[t][b]{Iteration}
\includegraphics[scale=0.4]{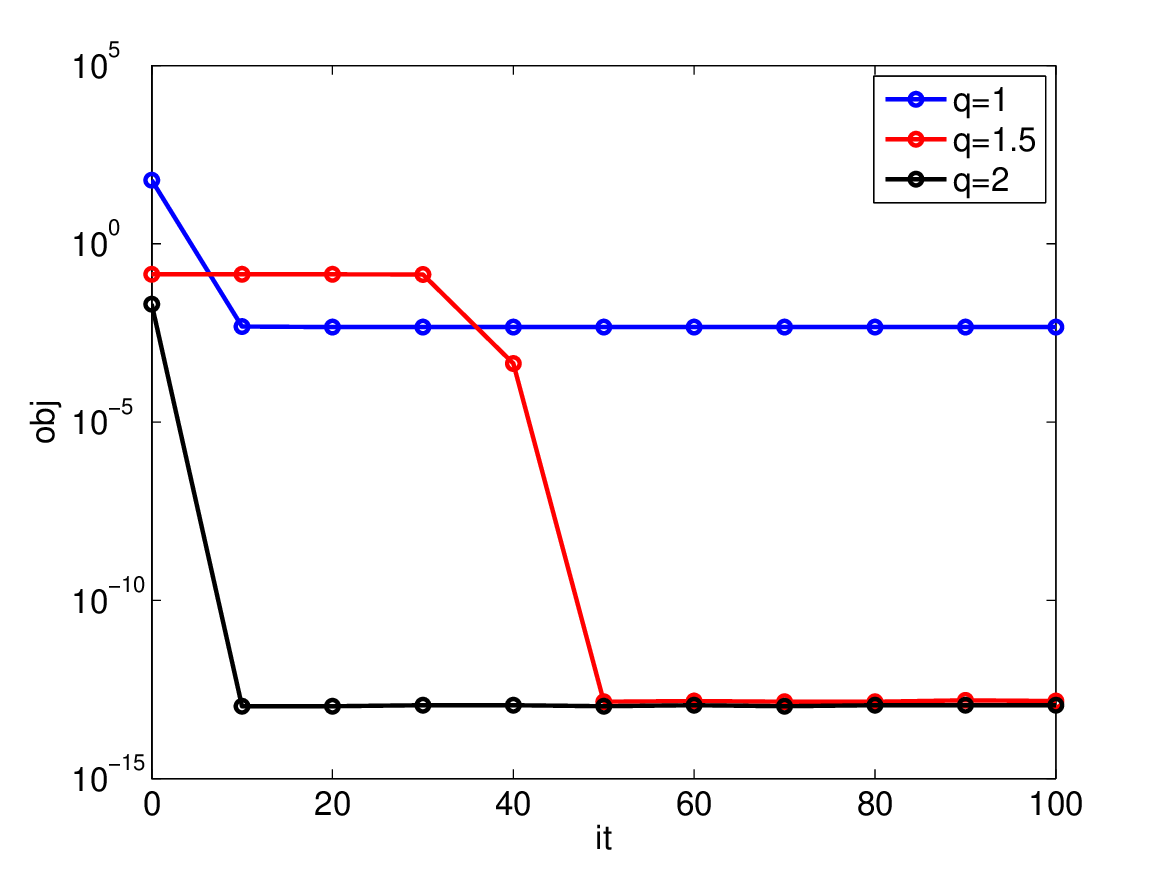} &
\psfrag{M}[b][t]{$M$}
\psfrag{it}[t][b]{Iteration}
\includegraphics[scale=0.4]{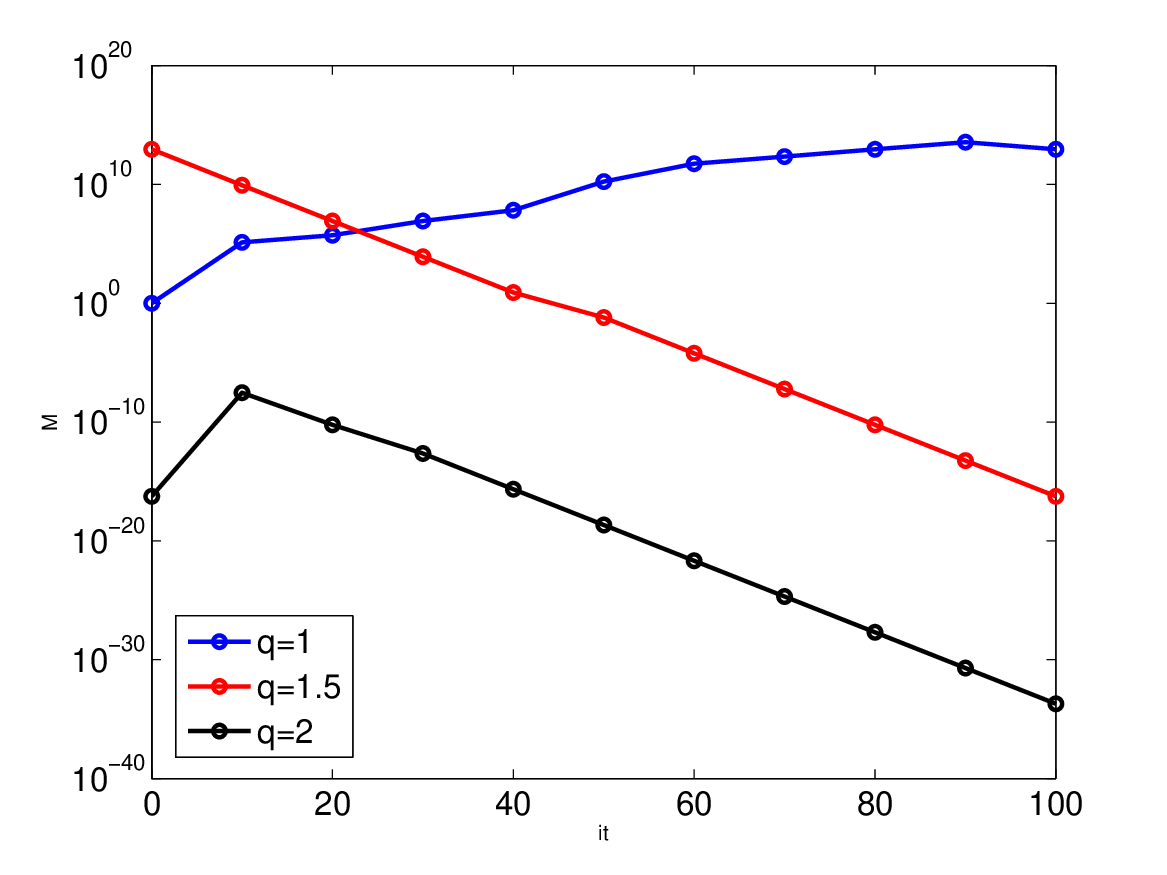}
\end{tabular}
\caption{We test the adaptivity of of Algorithm~\ref{method:adaptAccel}. {\em Left:} Convergence plot of Algorithm~\ref{method:adaptAccel} applied to the continuous Steiner problem~\eqref{eq:steiner} for $q=1, 1.5 , 2$. {\em Right:} Value of the local smoothness parameter $M$ across iterations. \label{fig:steiner}}
\end{center}
 \end{figure}

\section{Conclusion} From a practical point of view, the results above offer guidance in the choice of a prox function depending on the geometry of the feasible set $Q$. On the theoretical side, these results provide affine invariant descriptions of the complexity of an optimization problem based on both the geometry of the feasible set and of the smoothness of the objective function. In our first algorithm, this complexity bound is written in terms of the regularity constant of the polar of the feasible set and the Lipschitz constant of $\nabla f$ with respect to the Minkowski norm. In our last two methods, the regularity constant is replaced by a Bregman diameter constructed from an optimal choice of prox. 

When $Q$ is an $\ell_p$ ball, matching lower bounds on iteration complexity for the algorithm in \citep{Nest83} show that these bounds are optimal in terms of target precision, smoothness and problem dimension, up to a polylogarithmic term.

However, while we show that it is possible to formulate an affine invariant implementation of the optimal algorithm in \citep{Nest83}, we do not yet show that this is always a good idea outside of the $\ell_p$ case... In particular, given our choice of norm the constants $L_Q$ and $\Delta_Q$ are both affine invariant, with $L_Q$ optimal by our choice of prox function minimizing $\Delta_Q$ over all smooth square norms. However, outside of the cases where $Q$ is an $\ell_p$ ball, this does not mean that our choice of norm (Minkowski gauge of a centrally symmetric feasible set) minimizes the product $L_Q \min\{\Delta_Q/2,n\}$, hence that we achieve the best possible bound for the complexity of the smooth algorithm in \citep{Nest83} and its derivatives. Furthermore, while our bounds give clear indications of what an optimal choice of prox should look like, given a choice of norm, this characterization is not constructive outside of special cases like $\ell_p$-balls.

% TODO:
% What about projections or sections of l_p balls?
% Optimal symmetrization of sets
% Bounds on product for optimal choice of norm, link with curvature: C_f <= LD, is it tight?
% Algebra for regularity constants and sets. Discuss Calculus rules in J&N paper.
% AA: Use Renegar rounding result in Cor 4.3 in "Eff. FO for LP and SDP".
% AA: Numerical quadratic function over ellipsoid

\section*{Acknowledgments}  AA and MJ would like to acknowledge support from the European Research Council (project SIPA). MJ also acknowledges support from the Swiss National Science Foundation (SNSF). The authors would also like to acknowledge support from the chaire {\em \'Economie des nouvelles donn\'ees}, the {\em data science} joint research initiative with the {\em fonds AXA pour la recherche} and a gift from Soci\'et\'e G\'en\'erale Cross Asset Quantitative Research.

\small{\bibliographystyle{plainnat}
\bibsep 1ex
\bibliography{/Users/aspremon/Dropbox/Research/Biblio/MainPerso}}
\end{document}